\documentclass[reqno, 11pt]{amsart}

\usepackage{amsthm,amsmath,amssymb,amsfonts,bm}
\usepackage{stmaryrd}
\usepackage{bbold}
\usepackage{times}
\usepackage[all]{xy}
\usepackage{mathrsfs}
\usepackage{graphicx}
\usepackage{tikz}
\usetikzlibrary{matrix,calc,arrows}

\usepackage{comment}

\theoremstyle{plain}
\newtheorem{thm}{Theorem}
\newtheorem{lem}[thm]{Lemma}

\newtheorem{prop}[thm]{Proposition}

\theoremstyle{definition}
\newtheorem*{question}{Question}

\newtheorem{example}[thm]{Example}

\newtheorem*{notat}{Notations and conventions}

\theoremstyle{remark}
\newtheorem*{remark}{Remark}

\newcommand{\Ker}{\operatorname{Ker}} 

\newcommand{\Hom}{\operatorname{Hom}}
\newcommand{\End}{\operatorname{End}}

\newcommand{\Der}{\operatorname{Der}}

\newcommand{\ann}{\operatorname{ann}}

\newcommand{\Spec}{\operatorname{Spec}}
\newcommand{\Fract}{\operatorname{Fract}}

\newcommand{\onto}{\twoheadrightarrow}

\newcommand{\into}{\hookrightarrow}
\newcommand{\tto}{\longrightarrow}
\newcommand{\iso}{\stackrel{\sim}{\tto}}

\newcommand{\gen}[1]{\langle{#1}\rangle}

\renewcommand{\bar}[1]{\overline{#1}}

\newcommand{\til}[1]{\widetilde{#1}}

\newcommand{\Char}{\operatorname{char}}

\DeclareMathOperator{\Rat}{Rat}

\renewcommand{\k}{\mathbb{k}}
\renewcommand{\phi}{\varphi}

\newcommand{\HSpec}{H\text{-}\!\Spec}
\newcommand{\PSpec}{\mathcal{P}\text{-}\!\Spec}

\newcommand{\HRat}{H\text{-}\!\Rat}

\newcommand{\sE}{\mathscr{E}}

\newcommand{\cP}{\mathcal{P}}

\newcommand{\cH}{\mathcal{H}}
\newcommand{\cC}{\mathcal{C}}

\newcommand{\fg}{\mathfrak{g}}

\newcommand{\e}{\varepsilon}

\newcommand{\CC}{\mathbb{C}}

\newcommand{\cat}[1]{\operatorname{\mathsf{#1}}}

\newcommand{\ModAlg}[1]{{}_{#1}\!\cat{Alg}}

\newcommand{\cen}{\mathcal{Z}}
\newcommand{\Q}{\operatorname{Q}\!}
\newcommand{\bdot}{\,\text{\raisebox{-.45ex}{$\boldsymbol{\cdot}$}}\,}

\newcommand{\byH}{\!\!:\!\! H}

\newcommand{\EH}{\sE^{H}}
\newcommand{\QH}{\Q^{H}\!}
\newcommand{\CH}{\cC^{H}\!}

\newcommand{\ot}{\otimes}
\newcommand{\ant}{{\mathsf{S}}}
\newcommand{\rhk}{\text{\scriptsize $\rightharpoonup$}}

\newcommand{\upin}{\text{\ \rotatebox{90}{$\in$}\ }}
\newcommand{\upsub}{\text{\ \rotatebox{90}{$\subseteq$}\ }}
\newdir{ >}{{}*!/-5pt/\dir{>}}

\newcommand{\acts}{%
  \!\mathrel{\begin{tikzpicture}[scale=.8,baseline=(current  bounding  box.south)] 
  \useasboundingbox (-.6,-.2) rectangle (.1,.2);
  \node at (0,-.06) {} edge[out=210,in=150,loop] ();
\end{tikzpicture}}\!\!\!\!
}

\newcommand{\Clo}[1]{\mathscr{V}{#1}}

\newcommand{\HClo}[1]{\mathscr{V}^H{#1}}
\newcommand{\kH}{\kappa_H}
\newcommand{\kHrat}{\kappa_H^{\Rat}}

\begin{document}

\title[Hopf Algebra Actions and Rational Ideals]%
{Hopf Algebra Actions and Rational Ideals}

\author{Martin Lorenz}

\address{Department of Mathematics, Temple University,
    Philadelphia, PA 19122}

\email{lorenz@temple.edu}

\urladdr{http://www.math.temple.edu/$\stackrel{\sim}{\phantom{.}}$lorenz}



\maketitle

\begin{abstract}
This note discusses a framework for the investigation of the prime spectrum 
of an associative algebra $A$ that is equipped
with an action of a Hopf algebra $H$. 
In particular, we study a notion of $H$-rationality for ideals of $A$
and comment on a possible Dixmier-Moeglin equivalence for $H$-prime ideals of $A$.
\end{abstract}

\maketitle


\section*{Introduction}

\subsection{} \label{SS:action}
Actions of a group or Lie algebra on a given algebra $A$ have proven a useful tool in
analyzing its prime spectrum, $\Spec A$; see, e.g., \cite{kG06}, \cite{mL09}. 
The present article explores the more
general situation where a Hopf algebra $H$ acts on $A$. Our particular focus
will be on a notion of rationality for prime ideals of $A$ that takes the $H$-action
into account. Rational ideals were introduced into non-commutative
algebra by Dixmier in connection with his investigation of primitive ideals in enveloping algebras
of finite-dimensional Lie algebras. A highlight of this work is the celebrated Dixmier-Moeglin
equivalence \cite{jD77}, \cite{cM80}: rationality is equivalent to primitivity
and also to local closedness for primes of enveloping algebras.
In the context of arbitrary associative algebras, a notion of rationality was defined in \cite{mL08}.

\subsection{}
Let $A$ be an associative algebra (with $1$) over a field $\k$
and let $H$ be a Hopf $\k$-algebra.
A (left) \emph{$H$-action} on $A$ is a $\k$-linear map
$H \ot_\k A \to A$, $h\ot a \mapsto h.a$, that makes $A$ into a 
left $H$-module and satisfies 
$h.(ab) =  (h_1.a)(h_2.b)$ and $h.1 =  \gen{\e,h} 1$
for all $h\in H$and $a,b \in A$. Here, $\Delta h = h_1 \ot h_2$ 
is the comultiplication $\Delta \colon H \ot H \to H$ and $\e \colon H \to \k$ is the counit. 
We will write $H \acts A$ to indicate such an action.
An algebra $A$ that is equipped with an $H$-action is called an \emph{$H$-module algebra}.
With algebra maps that are also $H$-module maps as morphisms, $H$-module algebras
form a category, $\ModAlg H$. 

\subsection{}
\label{SS:Intro3}
Let $A \in \ModAlg H$. An ideal $I$ of $A$ that is also an $H$-submodule 
will be referred to as an \emph{$H$-ideal}. In this case, $A/I \in \ModAlg H$.
The sum of all $H$-ideals of $A$ that are
contained in a given arbitrary ideal $I$, clearly the unique maximal $H$-ideal of $A$ 
that is contained in $I$, will be called the \emph{$H$-core} of $I$ and
denoted by $I\byH$. Explicitly,
\begin{equation*} 
I\byH = \{ a \in A \mid H.a \subseteq I \}.
\end{equation*}
If $A \neq 0$ and the product of any two
nonzero $H$-ideals of $A$ is again nonzero, then $A$ is said to be \emph{$H$-prime}. 
An $H$-ideal $I$ of $A$ is called $H$-prime if $A/I$ is $H$-prime. 
We will denote the collection of all $H$-primes of $A$ 
by $\HSpec A$. It is easy to see that $H$-cores of prime ideals are $H$-prime. Thus,
we have a map,
\[
\kH = \bdot\byH \colon \Spec A \to \HSpec A . 
\]
A routine application of Zorn's Lemma shows that $\kH$ is surjective 
if the (two-sided) ideal of $A$ that is generated by $H.a$
is finitely generated for each $a \in A$ \cite[Exercise 10.4.4]{mL18}. This certainly holds
if $A$ is noetherian or the action $H \acts A$ is locally finite.
The fibers of $\kH$ are called the \emph{$H$-strata} of $\Spec A$:
\begin{equation*}
\Spec_I A:= \kH^{-1}(I) = \{ P \in \Spec A \mid P\byH = I\} \qquad (I \in \HSpec A).
\end{equation*}

\subsection{}
Rational and $H$-rational ideals are special prime and $H$-prime ideals, respectively,
that are of particular interest.
They are defined in Section~\ref{S:Hrat} below in terms of the symmetric ring of
quotients rather than the right Amitsur-Martindale quotient ring that was employed in \cite{mL08}. 
While this makes no essential difference, because the centers of these quotient rings are isomorphic
for semiprime rings \cite[E.3]{mL18}, 
symmetric quotient rings have some advantages such as their apparent
right-left symmetry. Sections~\ref{S:quotients} and \ref{S:Hcenter} deploy the requisite
background material on symmetric quotient rings. 
In Section~\ref{S:Hrat}, we show that the core map $\kH$
sends rational ideals to $H$-rational ideals. Thus, denoting the collections of 
rational and $H$-rational ideals of $A$ by $\Rat A$ and $\HRat A$, respectively,
we have a diagram of maps,
\begin{equation} 
\label{E:SpecDiagram}
\begin{tikzpicture}[baseline=(current  bounding  box.357), >=latex, scale=.7,
bij/.style={above,sloped,inner sep=0.5pt}]
\matrix (m) [matrix of math nodes, row sep=.1em,
column sep=4em, text height=1.5ex, text depth=0.25ex]
{\Spec A & \HSpec A \\ \upsub & \upsub \\ \Rat A & \HRat A \\};
\draw[->] (m-1-1) edge node[below]{\scriptsize $\kH$} (m-1-2);
\draw[->] (m-3-1) edge node[below]{\scriptsize $\kHrat$} (m-3-2);
\end{tikzpicture} 
\end{equation}
The notion of $H$-rationality and the diagram \eqref{E:SpecDiagram} 
have previously been explored in some special cases including the following.

\subsubsection{} 
\label{SSS:Ug}
If $H= U\fg$ is the enveloping algebra of a Lie $\k$-algebra $\fg$, then 
an $H$-module algebra structure on $A$
amounts to a Lie homomorphism $\fg \to \Der A$, the Lie algebra of
derivations of $A$. A special case of this setup arises from a \emph{Poisson structure} 
on $A$, that is,
a $\k$-bilinear map $\{\bdot\,,\bdot\} \colon A \times A \to A$ such that
$\fg:=(A,\{\bdot\,,\bdot\})$ is a Lie $\k$-algebra and $\{a,\bdot\} \in \Der A$ for all $a \in A$. 
We then have an action of $H= U\fg$ on $A$ that is determined by $x.a = \{x,a\}$ for $x\in \fg$, $a\in A$. 
In this setting, $H$-cores and $H$-primes are called \emph{Poisson cores} and
\emph{Poisson primes}, respectively, $\HSpec A$ is denoted by $\PSpec A$ or similar,
and $H$-rational ideals are called \emph{Poisson rational}.
See, for example,  \cite{kBiG03}, \cite{kG06}, \cite{BLSM17}. 

\subsubsection{}
\label{SSS:kG}
References \cite{mL08}, \cite{mL09} consider \eqref{E:SpecDiagram}
for a group algebra $H = \k G$, with particular emphasis on the case where $\k$ is algebraically closed
and $G$ is an affine algebraic $\k$-group that acts rationally on $A$.
In this setting, both maps in \eqref{E:SpecDiagram} are surjective and their fibers are known.
The fibers of $\kHrat$ are exactly the $G$-orbits in $\Rat A$. When $G$ is connected,
the strata $\Spec_I A$ can be described in terms of the spectra of certain
commutative algebras; more generally, this holds for the $H$-strata of
any ``integral'' action $H \acts A$ with $H$
cocommutative \cite{LNY20}.
If $G$ is an algebraic $\k$-torus, for example, then each
stratum is homeomorphic to the prime spectrum of a suitable commutative 
Laurent polynomial algebra over some $\k$-field \cite{mL13}.

\subsection{}
\label{SSS:DM}
The Dixmier-Moeglin equivalence ties, under suitable hypotheses, 
the representation-theo\-retic notion of  ``primitivity'' to the the field-theo\-retic
notion of ``rationality'' and to the topological notion of ``local closedness.''
Section~\ref{S:topology} describes a topology on $\HSpec A$,
a straightforward generalization of the familiar Jacobson-Zariski topology on $\Spec A$. 
This topology was considered earlier
for group actions, differential algebras, and in other settings;
see, e.g., \cite{mL09}, \cite{kG06}, \cite{kGeL00a}.
Aside from establishing a general framework, this section and the rest of this note offer 
little in the way of substantive results. 
Hopefully, the setup described here will lead to deeper 
investigations into the topological aspects of the
Dixmier-Moeglin equivalence, perhaps along the lines of the interesting work in \cite{jBxWdY19},  
or into some other possible avenues for future work that are pointed out below.


\begin{notat}
We will work over an arbitrary base field $\k$ and write
$\ot = \ot_\k$\,. The notations and
hypotheses introduced in the foregoing will remain in effect
for the remainder of this paper. In particular, $H$ will always be a Hopf $\k$-algebra
with antipode $\ant$ and counit $\e$, 
and $A$ will be a left $H$-module $\k$-algebra with action $H \acts A$
written as $h \ot a \mapsto h.a$\,. 
We assume throughout that the antipode $\ant$ is bijective.
\end{notat}


\section{Background on quotient rings} 
\label{S:quotients}

In this section, $R$ is an arbitrary ring (with $1$). We recall some basics concerning
the symmetric ring of quotients, $\Q R$. For details, see \cite[Appendix E]{mL18}. 
Throughout, we let 
\[
\sE = \sE(R)
\]
denote the collection of all (two-sided) ideals 
of $R$ having zero left and right annihilator in $R$.


\subsection{Symmetric quotient rings} 
\label{SS:Q}
 
The ring $R$ is a subring of $\Q R$. Moreover, the following hold for any $q \in \Q R$:
\begin{equation} 
\label{E:Dq}
D_q : = \{ r \in R \mid qRr \subseteq R \text{ and } rRq \subseteq R \} \in \sE.
\end{equation}
\begin{equation} 
\label{E:annQ}
I \in \sE, q \neq 0 \implies qI \neq 0 \text{ and } Iq \neq 0.
\end{equation}
For any $I \in \sE$, let $\Hom({}_RI,{}_RR)$ and  $\Hom(I_R,R_R)$ denote the collections of all left and right
$R$-module maps $I \to R$, respectively, and define
\begin{equation*}
\cH_I:= \big\{ (f,g) \in \Hom({}_RI,{}_RR) \times \Hom(I_R,R_R) \mid 
f(a)b = ag(b) \ \forall a,b \in I \big\}.
\end{equation*}
Writing $f(a) = af$ and $g(b) = gb$, the condition $f(a)b = ag(b)$ above becomes a variant of
associativity: $(af)b = a(gb)$. In fact,
\begin{equation} 
\label{E:H_I}
I \in \sE, (f,g) \in \cH_I \implies \exists q \in \Q R: aq = af, qb = gb \quad \forall a,b \in I.
\end{equation}
By \eqref{E:annQ}, the above $q$ is unique; we will write $q = [f,g] \in \Q R$\,.

\subsection{The extended center} 
\label{SS:C}

The center $\cC R:= \cen(\Q R)$ is called the \emph{extended center} of $R$; it
coincides with the centralizer of $R$ in $\Q R$:
\begin{equation}
\label{E:C(A)}
\begin{aligned}
\cC R &= \big\{ q \in \Q R \mid qr = rq \ \forall r \in R \big\} \\
&= \big\{ q \in \Q R \mid \exists I \in \sE: qa = aq \ \forall a \in I \big\} \,.
\end{aligned}
\end{equation}
In particular, $\cen R \subseteq \cC R$.
If $\cen R = \cC R$, then $R$ is called \emph{centrally closed}. In general, 
the following subring of $\Q R$ has center $\cC R$ and may be strictly larger than $R$:
\[
\til R:= R(\cC R) \subseteq \Q R.
\]
If $R$ is semiprime, then $\til R$ is a centrally closed ring
\cite[Theorem 3.2]{wBwM79}, called the \emph{central closure} of $R$.
If $R$ is a $\k$-algebra, then so are $\Q R$ and $\til{R}$, because $\cen R \subseteq \cC R = \cen \til{R}$.

\subsection{An extension lemma} 
\label{SS:extension}

We will need a version of
\cite[Lemma 4]{mL08} for symmetric rings of
quotients. The proof is essentially identical to the one in \cite{mL08}, 
but we include it here in full detail because of its somewhat technical nature.

Recall that a ring homomorphism $\phi \colon R \to S$ is called
\emph{centralizing} if the ring $S$ is generated by the image $\phi R$ and
the centralizer $C_S(\phi R) = \{ s \in S \mid s \phi(r) = \phi(r)s
\ \forall r \in R \}$. Any such $\phi$ maps the center
$\cen R$ to $\cen S$ and, for any ideal $I$ of $R$, the ideal of $S$
that is generated by $\phi I$ is given by $\gen{\phi I} := (\phi I)C_S(\phi R) = C_S(\phi R)(\phi I)$.

\begin{lem} 
\label{L:extension}
Let $\phi \colon R \to S$ be a centralizing ring homomorphism and let
\[
\cC_\phi = \{ q \in \cC R \mid \gen{\phi D_q} \in \sE(S)  \} .
\]
Then $R\cC_\phi$ is a subring of $\til R$ containing $R$. The map
$\phi$ extends uniquely to a homomorphism
$\til{\phi} \colon R\cC_\phi \to \til S$ which is centralizing. In particular,
$\til{\phi}\cC_\phi \subseteq \cC S$.
\end{lem}

\begin{proof}
Below, we put
$C:= C_S(\phi R)$ and $I_q:= \gen{\phi D_q} = (\phi D_q)C$  for $q \in \Q R$.
Let 
\[
R_\phi = \{ q \in \Q R \mid I_q \in \sE(S) \}.
\]
Since $1 \in D_q$ for $q \in R$, we certainly have $R
\subseteq R_\phi$. For  any $q,q' \in \Q R$, one checks that
$D_q \cap D_{q'} \subseteq D_{q+q'}$ and
$D_{q'}D_q \subseteq D_{qq'}$. Hence $I_{q'}I_q \subseteq I_{qq'} \cap I_{q+q'}$. If $I_q$ and $I_{q'}$ both belong to
$\sE(S)$, then $I_{q'}I_q \in \sE(S)$ and hence also $I_{q+q'}, I_{qq'} \in \sE(S)$. Thus, $R_\phi$
is a subring of $\Q R$. Note that $\cC_\phi = \cC R \cap R_\phi = 
\cen R_\phi$. So $R\cC_\phi$ is also a subring of
$\Q R$, containing $R$.

It remains to construct a ring homomorphism
$\til{\phi} \colon R\cC_\phi \to \til S$ that extends $\phi$.
Such an extension will necessarily be unique by \eqref{E:annQ}, because $\phi(qd) = \til\phi(q)\phi(d)$
for $q \in R\cC_\phi$, $d \in D_q$ and $I_q \in \sE(S)$.  

First, we define a ring homomorphism $\til{\phi} \colon \cC_\phi  \to \cC S$. 
To this end, let $q \in \cC_\phi$ be given. Define $f_q \colon I_q \to S$ by
\begin{equation*}
f_q(\sum_i \phi(d_i)c_i):= \sum_i \phi(qd_i)c_i \qquad (d_i \in D_q, c_i \in C).
\end{equation*}
To see that $f_q$ is well defined, let $d \in D_q$. Then
$\phi(dq)\sum_i \phi(d_i)c_i  = \sum_i \phi(dqd_i) c_i  = \phi(d) \sum_i \phi(q d_i)c_i$.
Thus, if $\sum_i \phi(d_i)c_i = \sum_j \phi(d'_j)c'_j$ with $d_i,d'_j \in D_q$ and $c_i,c'_j \in C$, then 
\[
0 = \phi(D_q q)\Big( \sum_i \phi(d_i)c_i - \sum_j \phi(d'_j)c'_j \Big)  
= \phi D_q \Big( \sum_i \phi(qd_i)c_i - \sum_j \phi(qd'_j)c'_j \Big)
\]
and so $\sum_i \phi(qd_i)c_i = \sum_j \phi(qd'_j)c'_j$ because $I_q \in \sE(S)$. Therefore, $f_q$ is well defined.
Next, we show that $f_q$ is an $(S,S)$-bimodule map. For $s = \phi(r)c \in S$ with $r \in R$, $c \in C$
and $x = \sum_i \phi(d_i)c_i \in I_q$\,, we compute using the fact that $qr=rq$,
\[
\begin{aligned}
f_q(sx) &= f_q( \sum_i \phi(rd_i)cc_i) = \sum_i \phi(qrd_i)cc_i \\
&= \sum_i \phi(rqd_i)cc_i = s\sum_i \phi(qd_i)c_i = s f_q(x).
\end{aligned}
\]
Similarly, $f_q(xs) = f_q(x)s$. So $f_q$ is indeed an $(S,S)$-bimodule map; equivalently, the pair $(f_q,f_q)$
belongs to the set $\cH_{I_q}$ in \eqref{E:H_I}. The element $t = [f_q,f_q] \in \Q S$ satisfies
$xt = f_q(x) = tx$ for all $x \in I_q$. By \eqref{E:C(A)}, it follows that $t \in \cC S$. Defining
$\til{\phi}(q):= t$ we obtain a map $\til\phi \colon \cC_\phi \to \cC S$. The definition of $\til{\phi}$ 
gives 
\[
\til{\phi}(q)\phi(d) = f_q(\phi(d)) = \phi(qd) \qquad (q\in \cC_\phi, d \in D_q).
\]
In particular, for $q,q' \in \cC_\phi$ and $d \in D_q$, $d' \in D_{q'}$, we obtain
$\til{\phi}(qq')\phi(d'd) = \phi(qq'd'd) = \til\phi(q)\phi(q'd'd) = \til\phi(q)\til\phi(q')\phi(d'd)$,
which implies $\til{\phi}(qq') = \til\phi(q)\til\phi(q')$ because $I_{q'}I_q \in \sE(S)$.
Similarly, one checks that $\til{\phi}(q+q') = \til\phi(q)+\til\phi(q')$; so $\til\phi$ is a ring
homomorphism.

Finally, define $\til{\phi} \colon R\cC_\phi \to \til S$ by 
\[
\til{\phi}(\sum_i r_i q_i) = \sum_i \phi(r_i)\til{\phi}(q_i) \qquad (r_i \in R, q_i \in \cC_\phi).
\]
For well definedness, assume $\sum_i r_i q_i = \sum_j r'_j q'_j$ and
let $d \in D = \bigcap_{i,j} (D_{q_i} \cap D_{q'_j})$. Then 
\begin{equation*}
\begin{aligned}
\Big( \sum_i \phi(r_i)\til{\phi}(q_i) - \sum_j \phi(r'_j)\til{\phi}(q'_j) \Big) \phi(d) 
&= \sum_i \phi(r_i)\phi(q_id) - \sum_j \phi(r'_j)\phi(q'_jd) \\
&= \sum_i \phi(r_iq_id) - \sum_j \phi(r'_jq'_jd) \\
&= \phi \Big(\big(\sum_i r_iq_i - \sum_j r'_jq'_j \big)d \Big) = 0.
\end{aligned}
\end{equation*}
Therefore, $\sum_i \phi(r_i)\til{\phi}(q_i) = \sum_j \phi(r'_j)\til{\phi}(q'_j)$ 
by \eqref{E:annQ}, because $\gen{\phi(D)} \in \sE(S)$. This proves well definedness 
of $\til\phi$. The fact that $\til\phi$ is a centralizing ring homomorphism
follows readily from the corresponding properties of $\phi$ and $\til\phi\big|_{\cC_\phi}$.
This completes the proof.
\end{proof}


\section{The extended $H$-center} 
\label{S:Hcenter}

In this section, we return to $H$-actions. We fix $A \in \ModAlg H$ and put $\sE = \sE(A)$. We also put
\[
\EH = \{ I \in \sE \mid I \text{ is an $H$-ideal} \}.
\]


\subsection{Extended $H$-invariants}
\label{SS:QH A}

It may not always be possible to extend the $H$-action on $A$ to an
action $H \acts \Q A$. Thus, the familiar definition of $H$-invariants,
\[
A^H = \{ a \in A \mid h.a = \gen{\e,h} a \text{ for all } h \in H \}, 
\]
is not directly applicable to $\Q A$ in this form. Instead, using the ideal $D_q \in \sE$ in \eqref{E:Dq}, we define
\[
\QH A := \big\{ q \in \Q A \mid h.(dq) = (h.d)q,  h.(qd) = q(h.d) \ \forall h \in H, d \in D_q \big\}.
\]
This makes sense, since the above equations only involve the $H$-action on elements of $A$. 

\begin{lem}
\label{L:QH A}
\begin{enumerate}
\item
$D_q \in \EH$ for any $q \in \QH A$. 
\item
Let $q \in \Q A$. If there is  some 
$I \in \EH$ such that $I \subseteq D_q$ and $h.(aq) = (h.a)q$, $h.(qa) = q(h.a)$ for all $h \in H$, $a \in I$, 
then $q \in \QH A$.
\item
$\QH A$ is a subalgebra of $\Q A$ with $A^H \subseteq \QH A$.
\end{enumerate}
\end{lem}

\begin{proof}
(a)
We will use the following identities in $A$; see \cite[Exercise 10.4.1]{mL18}:
\begin{equation}
\label{E:identity}
(h.a)b = h_1.\big( a(\ant(h_2).b)\big) ,\
a(h.b) = h_2.\big( (\ant^{-1}(h_1).a)b\big) 
\qquad (a,b \in A, h \in H). 
\end{equation}
Now let $q \in \QH A$ and $d \in D_q$ be given. Then, for any $a \in A$ and $h \in H$, 
\[
(h.d)aq = \big(h_1.\big( d(\ant(h_2).a)\big)\big)q = h_1.\big( d(\ant(h_2).a)q \big) ,
\]
where the second equality follows from $(h.d')q = h.(d'q)$ for all $h \in H$, $d' \in D_q$, because
$d(\ant(h_2).a) \in D_q$\,. The last expression belongs to $A$; so $(H.D_q)Aq \subseteq A$.
The inclusion $qA(H.D_q) \subseteq A$ follows similarly from the computation
\[
qa(h.d) = q h_2.\big( (\ant^{-1}(h_1).a)d \big) = h_2.\big( q(\ant^{-1}(h_1).a)d \big) .
\]
This shows that $H.D_q \subseteq D_q$\,.

\medskip

(b)
Let $q \in \Q A$ and $I \in \EH$ be as in the statement of (b). We need to check that 
$h.(dq) = (h.d)q$ and $h.(qd) = q(h.d)$ for $h \in H$, $d \in D_q$\,. For any $a \in I$, the first
identiy in \eqref{E:identity} gives
\[
(h.(dq))a \underset{\eqref{E:identity}}{=} h_1.\big( dq(\ant(h_2).a)\big)
= h_1.\big( d(\ant(h_2).(qa))\big) \underset{\eqref{E:identity}}{=} (h.d) qa.
\]
Thus, $(h.(dq) - (h.d) q)I = 0$ and so $h.(dq) = (h.d) q$ by \eqref{E:annQ}. The equality
$h.(qd) = q(h.d)$ is proved similarly.

\medskip

(c)
The defining conditions of $\QH A$ are readily checked for any $q \in A^H$, with $D_q = A$.
Thus, $A^H \subseteq \QH A$. To show that $\QH A$ is a subalgebra of $\Q A$, let $q,q' \in \QH A$
be given. Then the ideal  $I = D_q \cap D_{q'}$ belongs to $\EH$ and $I \subseteq D_{q+q'}$. 
Moreover, for any $a \in I$ and $h \in H$, one has 
\[
h.(a(q+q')) = h.(aq)+ h.(aq') = (h.a)q + (h.a)q' = (h.a)(q+q').
\]
Similarly, $h.((q+q')a) = (q+q')(h.a)$. In view of (b), it follows that $q+q' \in \QH A$.
The proof of $qq' \in \QH A$ is analogous, using $I  = D_{q'}D_q$\,.
\end{proof}

\begin{remark}
Lemma~\ref{L:QH A}(a) shows that $\QH A$ is contained in a certain subring of $\Q A$, the symmetric
ring of quotients for the ideal filter $\EH$; see \cite{mC86}, \cite{sM93a}. Denoting this ring of quotients by
$\Q_{\EH}A$, 
it has been shown in these references that $\Q_{\EH}A \in \ModAlg H$ via a 
unique extension of the action $H \acts A$ to $\Q_{\EH}A$. Hence, the algebra of $H$-invariants,
$(\Q_{\EH}A)^H$, is defined as usual. In fact,
\[
(\Q_{\EH}A)^H = \QH A.
\]
To see this, let $q \in (\Q_{\EH}A)^H$. Then $h.(dq) = (h_1.d)(h_2.q) = (h_1.d) \gen{\e,h_2}q = (h.d)q$
for all $h \in H$ and $d \in D_q$. Similarly, $h.(qd) = q(h.d)$; so $q \in \QH A$.
Conversely, let $q \in \QH A$. Then, for any $h \in H$ and $d \in D_q$, the first identity in \eqref{E:identity}
gives 
\[
(h.q)d = h_1.\big( q(\ant(h_2).d)\big) = h_1\ant(h_2).(qd) = \gen{\e,h} qd,
\] 
where the second equality uses that $q \in \QH A$ and $\ant(h_2).d \in D_q$ by Lemma~\ref{L:QH A}(a).
Therefore, $h.q = \gen{\e,h} q$ by \eqref{E:annQ} and so $q \in (\Q_{\EH}A)^H$.
\end{remark}

\begin{example}
If $H$ is pointed, then $\Q A \in \ModAlg H$ via a 
unique extension of the action $H \acts A$ to $\Q A$ \cite[2.3]{sM93a}. Hence
$(\Q A)^H$ is defined. As in the Remark, it follows that $\QH A = (\Q A)^H$.
\end{example}

\subsection{Extended $H$-centers}
\label{SS:CH A}

We define the \emph{extended $H$-center} of $A$ by
\[
\CH A = \cC A \cap \QH A .
\]

\begin{example}
If $H$ is cocommutative, then the center of any $H$-module algebra is $H$-stable  
\cite[Proposition 4]{mC86}. Thus,  if $H$ is pointed cocommutative (e.g., a group algebra or an
enveloping algebra), then Remark (2) above gives that
$\CH A = (\cC A)^H$, the subalgebra of $H$-invariants in the 
extended center of $A$.
\end{example}

\begin{example}
Let $A$ be a Poisson algebra.
With $H = U\fg$ as in \S\ref{SSS:Ug}, the algebra of $H$-invariants in
$A$ is called the \emph{Poisson center} of $A$ and usually denoted by $Z_{\cP}(A)$; so
\[
Z_{\cP}(A) = \{ a \in A \mid \{a,\bdot \} = 0 \}.
\]
If $A$ is a commutative domain, then $\Q A = \cC A = \Fract A$, the 
field of fractions of $A$, and $\CH A = Z_{\cP}(\Fract A)$. 
\end{example}

An important general property of $\CH A$ is stated in the following proposition the
first part of which is due to Matczuk \cite{jM91}.

\begin{prop} 
\label{P:CH A}
If $A$ is $H$-prime then $\CH A$ is a $\k$-field.
Conversely, if $A$
is semiprime and $\CH A$ is a field, then $A$ is $H$-prime.
\end{prop}

\begin{proof}
Assume that $A$ is $H$-prime and let $0 \neq q \in \CH A$ be given.
Recall from Lemma~\ref{L:QH A}(a) that $D_q \in \EH$.
Consequently, $I: = qD_q$ is a nonzero ideal of $A$ by \eqref{E:annQ}
and it is and $H$-ideal by definition of $\QH A$.
Since $A$ is $H$-prime, it follows that $I \in \sE$ \cite[Corollary 3]{mC86}.
Therefore, the map $f \colon D_q \to I$, $d \mapsto qd$ is
an isomorphism of $(A,A)$-bimodules. Thus, the pair $(f,f)$
belongs to the set $\cH_{I}$ in \eqref{E:H_I} and $q = [f,f]$. The desired
inverse of $q$ is given by $[f^{-1},f^{-1}]$.

Next, assume that $A$ is semiprime but not $H$-prime. Then there
exists a nonzero $H$-ideal $I$ of $A$ such that $J =
\ann_AI \neq 0$. By \cite[Corollary 2]{mC86}, $J$ is an $H$-ideal.
Since $A$ is semiprime, the sum $I + J$ is
direct and $I+J$ has zero annihilator; so $I + J \in \EH$. 
Define $(A,A)$-bimodule maps $f,f' \colon I + J \to A$
by $f(i+j) = i$ and $f'(i+j) = j$ and put $q = [f,f], q'=[f',f'] \in \cC A$\,.  
Since $f$ and $f'$ are $H$-equivariant, we also have $q,q' \in
\QH A$. Thus, $0 \neq q,q' \in \CH A$ but $qq' = 0$\,,
whence $\CH A$ is not a field.
\end{proof}


\section{Rationality and $H$-rationality} 
\label{S:Hrat}

We continue to assume that $A \in \ModAlg H$ throughout this section.


\subsection{Hearts of $H$-primes} 
\label{SS:hearts}

For any $I \in \HSpec A$, we define 
\[
\cH_H(I) = \CH(A/I);
\]
this is always a $\k$-field by Proposition~\ref{P:CH A}.
The $H$-prime $I$ will be called \emph{$H$-rational} if the field 
extension $\cH_H(I)/\k$ is algebraic. For a trivial $H$-action, we obtain the usual definitions:
the \emph{heart} $\cH(P) = \cC(A/P)$ of a prime $P \in \Spec A$, which is called
\emph{rational} when  $\cH(P)/\k$ is algebraic.
As in the Introduction, we will denote the collections of 
rational and $H$-rational ideals of $A$ by $\Rat A$ and $\HRat A$, respectively.

\begin{example}
Let $A$ be a commutative noetherian Poisson algebra
and assume that $\Char \k = 0$. Then each Poisson prime
$I \in \PSpec A$ is prime \cite{kG06}. Instead of of $\cH_H(I)$
with $H= U\fg$ (\S\ref{SSS:Ug}), one generally writes $\cH_{\cP}(I)$; so
$\cH_{\cP}(I) = Z_{\cP}(\Fract A/I)$. Moreover, $H$-rational ideals of $A$
are called \emph{Poisson rational}: $I \in \PSpec A$ is Poisson
rational iff the field extension $Z_{\cP}(\Fract A/I)/\k$ is algebraic.
\end{example}

For the special case of a group algebra $H = \k G$, the following result is \cite[Proposition 12]{mL08} and for 
commutative differential algebras, it was proved in \cite[Proposition 1.2]{kG06}; 
see also \cite[Lemma 3.4]{LWWxx}.

\begin{prop} 
\label{P:hearts}
Let $P \in \Spec A$\,. There is an embedding of $\k$-fields
$\cH_H(P\byH) \into \cH(P)$. 
In particular, if $P \in \Rat A$ then $P\byH \in \HRat A$.
\end{prop}

\begin{proof}
We may assume that $P\byH = 0$.
Thus $\CH A$ is a field and it suffices to
construct a $\k$-algebra map $\CH A \to \cC(A/P)$. For a given
$q \in \CH A$, we know that $D_q \in \EH$ by Lemma~\ref{L:QH A}(a). Therefore, $D_q \nsubseteq P$. 
Letting $\phi \colon A \onto A/P$ 
denote the canonical epimorphism, we have $\phi(D_q) \in \sE(A/P)$. Thus, we may apply 
Lemma~\ref{L:extension}, with $\CH A \subseteq \cC_\phi$,
and we obtain an extension $\til{\phi}$ of $\phi$ such that
$\til\phi(\CH A) \subseteq \cC(A/P)$. This is the desired homomorphism.
\end{proof}

Thus, we have a well-defined map $\kHrat = \kH\big|_{\Rat A} \colon \Rat A \to \HRat A$.
In contrast with the map $\kH \colon \Spec A \to \HSpec A$, which is often surjective, 
surjectivity of $\kHrat$ seems to require stronger hypotheses. 

\begin{example} \label{EX:nonsurjective}
Assume that $\Char \k = 0$ and let
$A = \k(x,y)$ be the rational function field over $\k$, equipped with the Poisson bracket 
that is determined by $\{x,y\} = x$. 
Of course, $A$ has no rational primes at all, yet it is not hard to see that $Z_{\cP}(A) = \k$; 
so the zero ideal is Poisson rational. On the other
hand, considering the polynomial algebra $\k[x,y]$ with the above Poisson bracket,
the zero ideal is still Poison rational, but now it is also the Poisson kernel of 
any maximal ideal of $\k[x,y]$ that does not contain $x$. Indeed, it is 
easy to see that all nonzero Poisson primes of $\k[x,y]$ contain $x$.
\end{example}


\subsection{Rational strata} \label{SS:strata}

It would be interesting to have a description of the $\kHrat$-fiber 
$\Rat_IA:= \Spec_IA \cap \Rat A = \{ P \in \Rat A \mid P\byH = I \}$ over a given $I \in \HRat A$.
The following result may serve as a first approximation. 
For any ideal $I$ of $A$, we may consider the convolution algebra $\Hom_\k(H,A/I)$
and the ``hit'' action $H \acts \Hom_\k(H,A/I)$ that is defined by 
$(h \rhk f)(k) = f(kh)$ for $h,k \in H$ and $f \in \Hom_\k(H,A/I)$ \cite[10.4.2]{mL18}. The map
\[
\alpha_I  \colon A \tto \Hom_\k(H,A/I), \quad a \mapsto (h \mapsto h.a + I)
\]
is a map in $\ModAlg H$ and $I\byH = \Ker\alpha_I$.

\begin{prop}
\label{P:fibres}
Given $I \in \HSpec A$, there is a bijection
\begin{center}
\begin{tikzpicture}[baseline=(current  bounding  box.center),  >=latex]
\node(11) at (-4,1.2){$\Bigl\{$%
$P \in \Spec A \mid A/P \cong \k \text{ \rm and } P\byH = I$%
$\Bigr\}$};
\node(12) at (4,1.2){$\Bigl\{$%
\begin{minipage}{1.5in}
\centering
\rm
embeddings $A/I \into H^*$ \\
in $\ModAlg H$
\end{minipage}%
$\Bigr\}$};
\node(31) at (-4,.5){$\upin$};
\node(32) at (4,.5){$\upin$};
\node(21) at (-4,0){$P$};
\node(22) at (4,0){$\alpha_P$};
\draw[<->]
(11) edge node[auto] {\tiny \rm bij.} (12)
(21) edge (22);
\end{tikzpicture}
\end{center}
\end{prop}

\begin{proof}
The set of all primes $P$ with 
$A/P \cong \k$ is in bijection with the set of all algebra maps
$A \to \k$\,.
Consider the isomorphism $\til{\phantom{x}} \colon \Hom_\k(A,\k) \iso \Hom_H(A,H^*)$
that is 
the composite of the canonical $\k$-linear isomorphisms
\[
\Hom_\k(A,\k) \iso \Hom_\k(H\ot_HA,\k) \iso \Hom_H(A,\Hom_\k(H,\k)) = \Hom_H(A,H^*)
\]
where the second isomorphism is $\Hom$-$\ot$ adjunction: $f \mapsto
f'$ with $f'(a)(h) = f(h\ot a)$; see, e.g., \cite[B.2.2]{mL18}. Explicitly,
\[
\til{\phi}(a)(h) = \phi(h.a) \qquad (\phi \in \Hom_\k(A,\k), h \in H, a \in A).
\]
If $\phi$ is an algebra map, then $\til{\phi}$ is an algebra map as well:
for $a,a' \in A$ and $h \in H$,
\[
\til{\phi}(aa')(h) = \phi(h.(aa'))  =  \sum \phi(h_1.a)\phi(h_2.a')
= \big( \til{\phi}(a)\til{\phi}(a') \big)(h);
\]
so $\til{\phi}(aa') = \til{\phi}(a)\til{\phi}(a')$. Conversely, if $\til{\phi}$ is an algebra map, then so is $\phi$:
\[
\phi(aa') = \til{\phi}(aa')(1) = (\til{\phi}(a)\til{\phi}(a'))(1) = \til{\phi}(a)(1)\til{\phi}(a')(1) = \phi(a)\phi(a'),
\]
Finally, $\Ker\til{\phi} = 
(\Ker\phi)\byH$. Thus, $\til{\phantom{x}}$ gives the desired bijection.
\end{proof}


\section{The topology of $\HSpec A$} 
\label{S:topology}

We continue to assume that $A \in \ModAlg H$.


\subsection{The Jacobson-Zariski topology}
\label{SS:JacZar}

The familiar Jacobson-Zariski topology on $\Spec A$ is defined by declaring all
subsets of the form $\Clo{S} = \{ P \in \Spec A \mid P \supseteq S \}$ for any subset $S \subseteq A$ to be closed
(e.g., \cite[1.3.4]{mL18}). In analogy with this definition,
we define the closed subsets of $\HSpec A$ to be those of the form
\[
\HClo S = \{ Q \in \HSpec A \mid Q \supseteq S \}.
\]
Evidently, $\HClo \varnothing = \HSpec A$, $\HClo{\{ 1\}} = \varnothing$, and 
$\HClo{\bigcup_\alpha  S_\alpha} = \bigcap_\alpha \HClo{S_\alpha}$ for
any family of subsets $S_\alpha \subseteq A$.
Since we may replace the set $S$ by the $H$-ideal of
$A$ that is generated by $S$ without changing $\HClo S$, the closed subsets of $\HSpec A$ can also
be described as the sets of the form $\HClo I$, where $I$ is an $H$-ideal of $A$.  
The defining property of $H$-prime ideals implies that $\HClo{I} \cup \HClo{J} = \HClo{IJ}$ for 
$H$-ideals $I$ and $J$. Thus, finite unions of closed sets are again closed,
thereby verifying the topology axioms. 
We list some of its basic properties in the following lemma.

\begin{lem}
\label{L:topology}
\begin{enumerate}
\item
The map $\kH \colon \Spec A \to \HSpec A$ is continuous.
\item
If all $H$-primes of $A$ are prime, then the topology of $\HSpec A$ is the initial
topology for the inclusion map $\HSpec A \into \Spec A$.
\item
Assume that $\kH$ is surjective and $\bigcap \Spec_I A  = I$
for all $I \in \HSpec A$. Then the topology of $\HSpec A$ is the final topology for $\kH$.
\end{enumerate}
\end{lem}

\begin{proof}
(a)
The preimage of the closed set $C = \HClo I$, for an $H$-ideal $I$ of $A$, is given by
\[
\kH^{-1}(C) = \{ P \in \Spec A \mid P\byH \supseteq I \} = \{ P \in \Spec A \mid P \supseteq I \} = \Clo I,
\]
which is closed in $\Spec A$ for the Jacobson-Zariski topology.
Therefore, $\kH$ is continuous.

\smallskip

(b)
By definition, the closed sets in the initial
topology for the inclusion $\HSpec A \into \Spec A$ are exactly the sets 
$\HSpec A \cap \Clo{S} = \HClo S$ \cite[Chap. 1 \S 2.3]{Bou71}.

\smallskip

(c)
A subset $C \subseteq \HSpec A$ is closed in the final topology for $\kH$, by definition, if and only if
$\kH^{-1}(C)$ is closed in $\Spec A$ \cite[Chap. 1 \S 2.4]{Bou71}. By (a), this
includes all sets of the form $C = \HClo I$ for an $H$-ideal $I$ of $A$. Conversely, assume that
$D:=\kH^{-1}(C) = \bigcup_{I \in C} \Spec_I A$ is closed in $\Spec A$; so $D = \Clo{J}$, where
$J = \bigcap D$. Note that $J = \bigcap_{I \in C} \bigcap \Spec_I A = \bigcap C$,
an $H$-ideal of $A$.
Since $\kH$ is assumed surjective, it follows that $C$ has the desired form:
\[
\begin{aligned}
C &= \kH(D) = \kH(\Clo J) = \{ P\byH \mid P \in \Spec A, P \supseteq J\} \\
& =  \{ P\byH \mid P \in \Spec A, P\byH  \supseteq J\} = \{ Q \in \HSpec A \mid  Q \supseteq J\} = \HClo J.
\qedhere
\end{aligned}
\]
\end{proof}

If $A$ has the maximum condition on ideals, then the condition $\bigcap \Spec_I A  = I$
for all $I \in \HSpec A$ in part (b) of the lemma is equivalent to the requirement that $H$-cores
of prime ideals of $A$ are semiprime. This requirement is satisfied for all group actions and for 
all actions of cocommutative Hopf algebras in characteristic $0$ \cite{LNY20}.


\subsection{Toward a Dixmier-Moeglin equivalence with $H$-action}
\label{SS:DM}

Let $I \in \HSpec A$. We will say that $I$ is \emph{$H$-locally closed} if the one-point set $\{ I \}$
is a locally closed subset in the topology of $\HSpec A$ (\S\ref{SS:JacZar}), that is,  $\{ I \}$ is
open in its closure $\bar{\{ I \}} = \HClo I$ or, equivalently, $\HClo I \setminus \{ I\}$ is a closed subset of 
$\HSpec A$. Explicitly, this means that
\[
I \ \subsetneqq \ \bigcap
\{ J \in \HSpec A \mid J \supsetneqq I \}.
\] 
Furthermore, $I$ will be called \emph{$H$-primitive} if 
$I = P\byH$ for some primitive ideal $P$ of $A$.

\smallskip

To summarize, the following three properties of $H$-prime ideals were considered in the foregoing:
(i) $H$-local closedness,  (ii) $H$-primitivity, and (iii) $H$-rationality. 
The following question has been studied in various settings before; see the examples below.
\begin{question}[Dixmier-Moeglin equivalence with $H$-action]
When are (i)--(iii) equivalent?
\end{question}
It turns out that the implications (i) $\Rightarrow$ (ii) $\Rightarrow$ (iii) hold under fairly general circumstances.
Indeed, a wide range of algebras $A$ shares the following two properties. 
\begin{description}
\item[\textbf{Weak Nullstellensatz}]
\index{Nullstellensatz!weak}%
The Schur division algebra
$\End_{A}(V)$ of very irreducible left $A$-module $V$ is algebraic over $\k$.
\item[\textbf{Jacobson Property}]
Every prime ideal of $A$ is an intersection of primitive ideals.
\end{description}
For example, both statements apply to
any countably generated noetherian algebra $A$ over an uncountable base field $\k$ and 
they also hold for many algebras over arbitrary fields; see \cite[II.7]{kBkG02} or \cite[Chapter 9]{jMcCjR87}.
The following lemma records some rather straightforward instances of the aforementioned 
implications.

\begin{lem}
\label{L:DMeasy}
\begin{enumerate}
\item
If $\kH$ is surjective and $A$ has the Jacobson property, then every $H$-locally closed $H$-prime ideal of $A$
is $H$-primitive.
\item
If $A$ satisfies the weak Nullstellensatz, then every $H$-primitive ideal of $A$ is $H$-rational.
\end{enumerate}
\end{lem}

\begin{proof}
(a)
Let $I \in \HSpec A$ be $H$-locally closed. By our hypotheses, $I = P\byH$ for some $P \in \Spec A$
and $P = \bigcap_\lambda Q_\lambda$ for some family $(Q_\lambda)$ of primitive ideals of $A$.
Since the core operator $\bdot\byH$ evidently commutes with intersections, we obtain 
$I = \bigcap_\lambda Q_\lambda\byH$. Finally, all $Q_\lambda\byH$ belong to $\HClo I$ and $I$ is 
locally closed. So we must have $I = Q_\lambda\byH$ for some $\lambda$ and, therefore, $I$ is $H$-primitive.

\smallskip

(b)
Assume that $I = P\byH$ for some primitive ideal $P$ of $A$; say $P$ is the annihilator of the
irreducible $A$-module $V$. Then the heart $\cH(P)$ embeds into the Schur division algebra $\End_A(V)$
\cite[Proposition E.2]{mL18}. Our hypothesis on $A$ now implies that $\cH(P)/\k$ is algebraic
and Proposition~\ref{P:hearts} further implies that $\cH_H(I)/\k$ is algebraic as well, showing that
$I$ is $H$-rational.
\end{proof}

Stronger hypotheses are needed to ensure the validity of (iii) $\Rightarrow$ (i).

\begin{example}
\label{EX:Ug}
By classical results of Hilbert, the ordinary Dixmier-Moeglin equivalence, without $H$-action, holds for primes
of any affine commutative $\k$-algebra, with ``primitive''
being the same as ``maximal.''
For an affine commutative Poisson algebra $A$ and $H = U\fg$ as in \ref{SSS:Ug}, the 
above Question was originally posed in \cite{kBiG03} (for $\k = \CC$);
the equivalence is known as the \emph{Poisson Dixmier-Moeglin equivalence} in this setting. 
Lemma~\ref{L:DMeasy} covers the easy implications, (i) $\Rightarrow$ (ii) $\Rightarrow$ (iii).
It was shown in \cite{BLSM17} that (iii) $\Rightarrow$ (ii) also holds---so (ii) and (iii) are in 
fact equivalent---but (iii) $\Rightarrow$ (i) can fail if the Krull dimension of $A$ is at least $4$.
\end{example}

\begin{example}
\label{EX:kG}
Let $G$ be an affine algebraic group over an algebraically closed field $\k$ 
that acts rationally on the $\k$-algebra $A$ (\ref{SSS:kG}). 
Assume that $A$ has the Jacobson property and satisfies the weak Nullstellensatz.
Then, again, Lemma~\ref{L:DMeasy} gives 
(i) $\Rightarrow$ (ii) $\Rightarrow$ (iii) with $H = \k G$.  It is also know in this setting, that every $G$-rational 
ideal of $A$ has the form $P\byH$ for some $P \in \Rat A$ and that $P$ is locally closed if and only if
$P\byH$ is $H$-locally closed \cite{mL09}.
Thus, if $A$ satisfies the ordinary Dixmier-Moeglin equivalence, and hence all rational primes are
locally closed, then the Question above has a positive answer
for $A$ and $H = \k G$.
\end{example}


\bibliographystyle{plain}

\def\cprime{$'$}


\end{document}